\documentclass[12pt,a4paper]{amsart}
\usepackage[utf8]{inputenc}
\usepackage{amsfonts}
\usepackage{amsmath, amssymb}
\usepackage{amsthm}
\usepackage{cases}
\usepackage[margin=.8in]{geometry}
\usepackage{tkz-euclide}

\newtheorem{theorem}{Theorem}[section]

\newtheorem{defin}[theorem]{Definition}
\newtheorem{lemma}[theorem]{Lemma}

\newtheorem{rem}[theorem]{Remark}

\usepackage{hyperref}
\hypersetup{
  colorlinks = true, 
  urlcolor = blue, 
  linkcolor = blue, 
  citecolor = blue 
}

\let \div \relax

\DeclareMathOperator{\div}{div}

\DeclareMathOperator{\dist}{dist}

\DeclareMathOperator{\R}{\mathbb{R}}

\providecommand{\Test}{\mathcal{D}}

\newcommand{\dd}{\, \mathrm{d}}
\newcommand{\del}{\partial}
\newcommand{\e}{\varepsilon}
\newcommand{\eps}{\varepsilon}
\newcommand{\weak}{\rightharpoonup}
\newcommand{\B}{\mathcal{B}}

\newcommand{\ms}{\mathbb{S}}
\newcommand{\vc}{\mathbf}
\newcommand{\vu}{\vc u}
\newcommand{\tvu}{\tilde{\vu}}
\newcommand{\tvr}{\tilde{\rho}}

\newcommand{\cre}{\color{black}}

\renewcommand{\rho}{\varrho}

\let\temp\phi
\let\phi\varphi
\let\varphi\temp

\allowdisplaybreaks

\title[Homogenization of NSE for lower adiabatic exponent]{Homogenization of the unsteady compressible Navier-Stokes equations for adiabatic exponent $\gamma>3$}
\date{\today}

\author{Florian Oschmann}
\address{Institute of Mathematics, Czech Academy of Sciences,
\v Zitn\'a 25, 115 67 Praha 1, Czech Republic.}
\email{oschmann@math.cas.cz}

\author{Milan Pokorn\'y}
\address{Charles University, Faculty of Mathematics and Physics, Mathematical Institute of  Char\-les University, Sokolovsk\'a 83, 186 75 Praha 8, Czech Republic.}
\email{pokorny@karlin.mff.cuni.cz}

\begin{document}
\maketitle

\begin{abstract}
We consider the unsteady compressible Navier-Stokes equations in a perforated three-di\-men\-sional domain, and show that the limit system for the diameter of the holes going to zero is the same as in the perforated domain provided the perforations are small enough. The novelty of this result is the lower adiabatic exponent $\gamma>3$ instead of the known value $\gamma>6$. The proof is based on the use of two different restriction operators {\cre leading to} two different types of pressure estimates. We also discuss the extension of this result for the unsteady Navier-Stokes-Fourier system as well as the optimality of the known results in arbitrary space dimension for both steady and unsteady problems.
\end{abstract}

\section{Introduction}

The homogenization in mathematical fluid mechanics and thermodynamics is usually connected with the sequence of problems studied in perforated domains containing many holes with a small radius approaching zero. Based on the relation of the size of the holes and their number, the question is which problem is satisfied in the limit on the domain without holes.\\

The first studies of this type were indeed performed for the steady \emph{incompressible} Stokes and Navier-Stokes equations, where the problem in a fixed bounded domain is well understood. Based on the original work of L.~Tartar \cite{Tartar1980}, G.~Allaire achieved the full picture for this problem in \cite{Allaire1989}, \cite{Allaire1990a}, and \cite{Allaire1990b}. Assuming the number of holes is of order $\e^{-3}$ {\cre for some $\e > 0$} and their diameter of order $\e^{\alpha}$, if $1\leq \alpha <3$ which corresponds to the case of large holes, the limit problem {\cre as $\e \to 0$} is the classical Darcy's law; for $\alpha >3$ corresponding to tiny holes, the limit problem remains the same as the original one (Stokes or Navier-Stokes equations depending on the sequence of problems). Finally, for the critical case $\alpha=3$, the limit problem is the Brinkman equations, that is, the original problem with an extra damping term. The same problems in the evolutionary case were studied by A.~Mikeli\'c \cite{Mikelic1991}, E.~Feireisl, Y.~Namlyeyeva, and \v{S}.~Ne\v{c}asov\'a \cite{FeireislNamlyeyevaNecasova2016}, and the picture was completed by Y.~Lu and P.~Yang in \cite{LuYang2023}.\\

The case of \emph{compressible} fluid flow is more complex and the complete picture has not been achieved yet. Assuming an adiabatic pressure law of the form $p(\rho) \sim \rho^\gamma$, the case $\alpha=1$ for the compressible Navier-Stokes equations was studied by N.~Masmoudi in \cite{Masmoudi2002}. The same situation for heat conducting fluids was investigated by E.~Feireisl, A.~Novotn\'y, and T.~Takahashi \cite{FeireislNovotnyTakahashi2010}. The problems for small holes, i.e., the limit problem is the same as the original one, was for the steady compressible Navier-Stokes system and an adiabatic exponent $\gamma>2$ given by L.~Diening, E.~Feireisl, and Y.~Lu in \cite{DieningFeireislLu2017}, and the evolutionary case was considered for $\gamma >6$ by Y.~Lu and S.~Schwarzacher in \cite{LuSchwarzacher2018}. Similar results for the heat conducting case were obtained by Y.~Lu and M.~Pokorn\'y in the steady case (see \cite{LuPokorny2021}), and by M.~Pokorn\'y and E.~Sk\v{r}\'\i \v{s}ovsk\'y for the evolutionary system (see \cite{PokornySkrisovski2021}). All the aforementioned results concern the case of three spatial dimensions. The situation in two space dimensions is more complex and has been recently resolved by \v{S}.~Ne\v{c}asov\'a and J.~Pan in the steady case (see \cite{NecasovaPan2022}), and by \v{S}.~Ne\v{c}asov\'a and F.~Oschmann in the evolutionary case (see \cite{NecasovaOschmann2023}), for $\gamma >1$ and $\gamma >2$, respectively. The last result is based on the ideas of M.~Bravin for the case of flow around an obstacle (see \cite{Bravin2022}). Randomly distributed holes were considered by A.~Giunti and R.~H\"ofer in \cite{GiuntiHoefer2019}, as well as in the thesis by F.~Oschmann \cite{OschmannDiss2022}, see also \cite{Oschmann2022} and \cite{BellaOschmann2023}.\\

In this paper, we want to focus on the value of the adiabatic exponent $\gamma$ in the pressure law $p(\rho) \sim \rho^\gamma$. The \emph{physical} values are in the range $1 \leq \gamma \leq 1 + \frac{2}{d}$, where the space dimension $d \in \{2,3\}$ (see, e.g., \cite{Clausius1857}). However, as mentioned above, the \emph{mathematically} known values are $\gamma > d-1$ in the steady case, and $\gamma>4d - 6$ for the evolutionary systems. Our aim is to lower the last bound to $\gamma>d$ by giving two different proofs. Moreover, we discuss why this bound seems to be optimal in terms of the dimension $d$.\\

\paragraph{\textbf{Notation.}} We use the standard notations for Lebesgue and Sobolev spaces, and denote them even for vector- or matrix-valued functions as in the scalar case, e.g., we use $L^p(D)$ instead of $L^p(D;\R^3)$. The Frobenius inner product of two matrices $A,B\in \R^{3\times 3}$ is denoted by $A:B=\sum_{i,j=1}^3 A_{ij} B_{ij}$. Moreover, we use the notation $a\lesssim b$ whenever there is a generic constant $C>0$ which is independent of $a$, $b$, and $\e$ such that $a\leq C b$. Lastly, we denote for a function $f$ with domain of definition $G \subset\R^3$ its zero prolongation by $\tilde{f}$, that is,
\begin{align*}
\tilde{f}=f \text{ in } G,\quad \tilde{f}=0 \text{ in } \R^3\setminus G.
\end{align*}

\paragraph{\textbf{Organization of the paper.}} The paper is organized as follows. In Section~\ref{sec2}, we introduce the compressible Navier-Stokes equations and the underlying domain, and formulate our main result. In Section~\ref{sec:Bds} we give uniform bounds on the functions as well as a refined pressure decomposition crucial in our analysis. Sections~\ref{sec:Conv} and \ref{sec:NSF} are devoted to show the convergence result for the Navier-Stokes and Navier-Stokes-Fourier equations, respectively. A different approach and its main ideas for the proof of the main result are given in Section~\ref{sec:Bravin}. Finally, in Section~\ref{sec:Concl} we discuss the optimality of the adiabatic exponent in terms of the space-time dimension of the problem.

\section{The model, weak solutions, and the main result}\label{sec2}
In this section, we introduce the perforated domain, the evolutionary compressible Navier-Stokes equations, and state our main result. We start with the description of the perforated domain and the equations governing the fluid's motion.
\subsection{The perforated domain and the Navier-Stokes equations}
For $\e\in (0,1)$, let $D\subset \R^3$ be a bounded domain with smooth boundary, and let $K_\e \subset D$ be a compact set. Define now
\begin{align}
D_\e= D\setminus K_\e. \label{defDeps}
\end{align}
Moreover, we assume that there exists a family of balls $B_{\e^\alpha}(x_i(\e))$, $i=1,...,N(\e)$, $\alpha\geq 1$, such that
\begin{align}
\begin{split}
K_\e \subset \bigcup_{i=1}^{N(\e)} B_{\e^\alpha}(x_i(\e)),\\
\dist(x_i(\e), \del D) > \e,\\
\forall i\neq j: |x_i(\e)-x_j(\e)| > 2 \e.
\end{split} \label{defKeps}
\end{align}
Note that this implies
\begin{align*}
|K_\e| \lesssim N(\e) \e^{3\alpha} \lesssim \e^{3(\alpha-1)}.
\end{align*}
For fixed $T>0$, we consider in $(0,T)\times D_\e$ the evolutionary compressible Navier-Stokes equations
\begin{align}\label{NSE}
\begin{cases}
\del_t \rho_\e + \div(\rho_\e \vu_\e)=0 & \text{in } (0,T)\times D_\e,\\
\del_t(\rho_\e \vu_\e) + \div(\rho_\e \vu_\e \otimes \vu_\e) + \nabla p(\rho_\e) = \div \ms(\nabla \vu_\e) + \rho_\e \vc f & \text{in } (0,T)\times D_\e,\\
\vu_\e=0 & \text{on } (0,T)\times \del D_\e,\\
\rho_\e(0,\cdot)=\rho_{\e, 0},\ (\rho_\e\vu_\e)(0,\cdot)=\vc m_{\e, 0} & \text{in } D_\e.
\end{cases}
\end{align}
Here, $\rho_\e$ and $\vu_\e$ denote the fluid's density and velocity, respectively, $p(s)=s^\gamma$ for some $\gamma>\frac32$, $\ms(\nabla \vu)$ is the Newtonian viscous stress tensor of the form
\begin{align*}
\ms(\nabla\vu)=\mu \Big(\nabla\vu + \nabla\vu^T-\frac23 \div(\vu)\mathbb{I} \Big)+\eta\div(\vu)\mathbb{I}, \quad \mu>0,\, \eta\geq 0,
\end{align*}
and $\vc f\in L^\infty((0,T)\times D)$ is given. The exact range of the adiabatic exponent $\gamma$ we can handle will be specified in Theorem~\ref{thm1} below.

\begin{rem}
As a matter of fact, we are able to consider pressure laws $p(s)$ satisfying
\begin{align*}
p\in C([0,\infty)) \cap C^1((0,\infty)),\quad p(0) = 0, \quad p'(s)>0 \ (s > 0), \quad \lim_{s\to \infty} \frac{p'(s)}{s^{\gamma-1}} = a > 0,
\end{align*}
however, we don't want to unnecessarily complicate the analysis.
\end{rem}

\subsection{Weak solutions and main result}
For further use, we introduce the concept of finite energy weak solutions.
\begin{defin}\label{def1}
Let $T>0$ be fixed, $\gamma>\frac32$, and let the initial data satisfy
\begin{align*}
\rho(0,\cdot)=\rho_0,\quad (\rho\vu)(0,\cdot)=\vc m_0,
\end{align*}
together with the compatibility conditions
\begin{equation}\label{init}
\begin{gathered}
\rho_0\geq 0 \text{ a.e.~in } D_\e, \quad \rho_0\in L^\gamma(D_\e),\\
\vc m_0=0 \text{ on } \{\rho_0=0\}, \quad \vc m_0 \in L^\frac{2\gamma}{\gamma+1}(D_\e),\quad \frac{|\vc m_0|^2}{\rho_0} \in L^1(D_\e).
\end{gathered}
\end{equation}
 We call a duplet $(\rho,\vu)$ a \emph{renormalized finite energy weak solution} to system \eqref{NSE} if:
\begin{itemize}
\item The solution belongs to the regularity class
\begin{gather*}
\rho\geq 0 \text{ a.e.~in } (0,T)\times D_\e, \quad \rho\in L^\infty(0,T;L^\gamma(D_\e)), \quad \int_{D_\e} \rho \dd x = \int_{D_\e} \rho_0 \dd x,\\
\vu\in L^2(0,T;W_0^{1,2}(D_\e)), \quad \rho\vu\in L^\infty(0,T;L^\frac{2\gamma}{\gamma+1}(D_\e));
\end{gather*}
\item We have
\begin{equation}\label{renCE}
\begin{split}
\del_t \rho + \div(\rho \vu) &= 0 \text{ in } \mathcal{D}'((0,T)\times D_\e),\\
\del_t \tilde{\rho} + \div(\tilde{\rho} \tilde{\vu}) &= 0 \text{ in } \mathcal{D}'((0,T)\times \R^3), \\
\del_t b(\tilde{\rho})+\div(b(\tilde{\rho})\tilde{\vu})+(\tilde{\rho}b'(\tilde{\rho})-b(\tilde{\rho}))\div\tilde{\vu} &= 0 \text{ in } \mathcal{D}'((0,T)\times \R^3)
\end{split}
\end{equation}
for any $b\in C([0,\infty)) \cap C^1((0,\infty))$ satisfying
\begin{align*}
b'(s) \leq c s^{-\lambda_0} \ \text{for} \ s \in (0,1],\quad b'(s) \leq c s^{\lambda_1} \ \text{for} \ s \in [1,\infty)
\end{align*}
with
\begin{align*}
c>0, \quad \lambda_0<1, \quad -1 < \lambda_1 \leq \frac12 \bigg(\frac53 \gamma - 1 \bigg) - 1;
\end{align*}
\item For any $\phi\in C_c^\infty([0,T)\times D_\e; \R^3)$,
\begin{align}\label{wkMom}
\begin{split}
&\int_0^T\int_{D_\e} \rho \vu \cdot \del_t \phi \dd x \dd t + \int_0^T\int_{D_\e} \rho \vu\otimes \vu : \nabla \phi \dd x \dd t + \int_0^T\int_{D_\e} \rho^\gamma\div \phi \dd x \dd t \\
&- \int_0^T\int_{D_\e} \ms(\nabla \vu):\nabla \phi \dd x \dd t + \int_0^T\int_{D_\e} \rho \vc f \cdot \phi \dd x \dd t = -\int_{D_\e} \vc m_0 \cdot \phi(0,\cdot) \dd x;
\end{split}
\end{align}

\item For almost any $\tau \in [0,T]$, the energy inequality holds:
\begin{align}\label{EI}
\begin{split}
&\int_{D_\e} \frac12 \rho |\vu|^2(\tau,\cdot) + \frac{\rho^\gamma(\tau,\cdot)}{\gamma-1} \dd x + \int_0^\tau \int_{D_\e} \ms(\nabla\vu):\nabla\vu \dd x \dd t \\
&\quad \leq \int_{D_\e} \frac{|\vc m_0|^2}{2\rho_0} + \frac{\rho_0^\gamma}{\gamma-1} \dd x + \int_0^\tau \int_{D_\e} \rho \vc f\cdot\vu \dd x \dd t.
\end{split}
\end{align}
\end{itemize}
\end{defin}

Regarding existence of weak solutions, we have the following
\begin{theorem}[{\cite[Theorem~1.1]{FeireislNovotnyPetzeltova2001}}]
Let $D_\e\subset \R^3$ be a bounded domain with smooth boundary, $\gamma>\frac32$, $T>0$ be given. Let the initial data satisfy \eqref{init}. Then, there exists a renormalized finite energy weak solution $(\rho,\vu)$ to system \eqref{NSE} in the sense of Definition~\ref{def1}.
\end{theorem}

We are now in the position to state our main result in this paper.
\begin{theorem}\label{thm1}
Let $D\subset\R^3$ be a bounded domain with smooth boundary, $K_\e \subset D$ comply with \eqref{defKeps}, and $D_\e$ be defined as in \eqref{defDeps}. Let $(\rho_\e,\vu_\e)$ be a sequence of renormalized finite energy weak solutions to system \eqref{NSE} emanating from the initial data $(\rho_{\e, 0}, \vc m_{\e, 0})$, and assume
\begin{align}\label{initConv}
\tilde{\rho}_{\e, 0}\to \rho_0 \text{ weakly in } L^\gamma(D), \quad  \frac{|\tilde{\vc m}_{\e, 0}|^2}{\tilde{\rho}_{\e, 0}} \to \frac{|\vc m_0|^2}{\rho_0} \text{ weakly in } L^1(D).
\end{align}
Then, there exists a subsequence (not relabelled) such that
\begin{align*}
\tilde{\rho}_\e &\weak^* \rho \text{ weakly-$\ast$ in } L^\infty(0,T;L^\gamma(D)) \text{ and weakly in } L^{\frac53 \gamma - 1}((0,T) \times D),\\
\tilde{\vu}_\e &\weak \vu \text{ weakly in } L^2(0,T;W_0^{1,2}(D)),
\end{align*}
where $(\rho,\vu)$ is a solution to system \eqref{NSE} in the domain $(0,T)\times D$ with initial conditions $\rho(0,\cdot)=\rho_0$ and $(\rho\vu)(0,\cdot)=\vc m_0$, provided
\begin{align}\label{cond1}
\gamma>3 \ \text{and} \ \alpha>\max\left\{ 3, \frac{2\gamma-3}{\gamma-3} \right \}. 
\end{align}
\end{theorem}

\begin{rem}
We remark that \eqref{cond1} with $\gamma \geq 6$ yields the $\gamma$-independent bound $\alpha>3$, which is sharper than the known bound $\alpha> \frac{3(2\gamma-3)}{\gamma-6} > 6$ in \cite{LuSchwarzacher2018}.
\end{rem}

\section{Uniform bounds}\label{sec:Bds}
In this section, we show uniform in $\eps$ bounds on the velocity, density, and momentum. Additionally, we give a refined pressure decomposition.
\subsection{Bounds obtained from the energy inequality}
\begin{lemma}
Under the assumptions of Theorem~\ref{thm1}, we have
\begin{align}\label{unifbds}
\|\rho_\e\|_{L^\infty(0,T;L^\gamma(D_\e))} + \|\sqrt{\rho_\e}\vu_\e\|_{L^\infty(0,T;L^2(D_\e))} + \|\vu_\e\|_{L^2(0,T;W_0^{1,2}(D_\e))} \leq C
\end{align}
for some constant $C>0$ independent of $\e$.
\end{lemma}
\begin{proof}
By the energy inequality \eqref{EI} and the assumptions on the initial data \eqref{initConv}, we obtain
\begin{align*}
\int_{D_\e} \frac12 \rho_\e |\vu_\e|^2(\tau,\cdot) + \frac{\rho_\e^\gamma(\tau,\cdot)}{\gamma-1} \dd x + \int_0^\tau \int_{D_\e} \ms(\nabla\vu_\e):\nabla\vu_\e \dd x \dd t \leq C + \int_0^\tau\int_{D_\e} \rho_\e \vc f \cdot \vu_\e \dd x \dd t.
\end{align*}
Note further that the conservation of mass and the convergence of the initial data $\rho_{\e, 0}$ yields
\begin{align*}
\|\rho_\e\|_{L^\infty(0,T;L^1(D_\e))}=\|\rho_{\e, 0}\|_{L^1(D_\e)}\leq |D_\e|^{1-\frac{1}{\gamma}} \|\tilde{\rho}_{\e, 0}\|_{L^\gamma(D)}\leq C
\end{align*}
since $|D_\e|\leq |D|$. Using now H\"older's and Young's inequality, we get for almost any $\tau\in [0,T]$
\begin{align*}
\int_{D_\e} \rho_\e\vc f \cdot \vu_\e(\tau) \dd x \dd t \leq C \|\rho_\e(\tau)\|_{L^1(D_\e)}^\frac12 \|\rho_\e |\vu_\e|^2(\tau)\|_{L^1(D_\e)}^\frac12 \leq C + \frac12 \|\rho_\e |\vu_\e|^2(\tau)\|_{L^1(D_\e)}.
\end{align*}
Thus, we end up with the inequality
\begin{align*}
\int_{D_\e} \frac12 \rho_\e |\vu_\e|^2(\tau,\cdot) + \frac{\rho_\e^\gamma(\tau,\cdot)}{\gamma-1} \dd x + \int_0^\tau \int_{D_\e} \ms(\nabla\vu_\e):\nabla\vu_\e \dd x \dd t \leq C + \int_0^\tau \int_{D_\e} \frac12 \rho_\e |\vu_\e|^2 \dd x \dd t.
\end{align*}
Using Gr\"onwall's inequality, we conclude that
\begin{align*}
\sup_{t \in (0,T)}\int_{D_\e} \frac12 \rho_\e |\vu_\e|^2(t,\cdot) + \frac{\rho_\e^\gamma(t,\cdot)}{\gamma-1} \dd x + \int_0^T \int_{D_\e} \ms(\nabla\vu_\e):\nabla\vu_\e \dd x \dd t \leq C(T).
\end{align*}
Since we may extend both $\rho_\e$ and $\vu_\e$ by zero to $D$ without influencing the inequality, by virtue of the Korn and Poincar\'e inequalities we conclude easily.
\end{proof}
Let us moreover remark that the bounds \eqref{unifbds} immediately imply for the linear momentum
\begin{align*}
\|\rho_\e \vu_\e\|_{L^\infty(0,T;L^\frac{2\gamma}{\gamma+1}(D_\e))} &= \|\sqrt{\rho_\e} \sqrt{\rho_\e} \vu_\e\|_{L^\infty(0,T;L^\frac{2\gamma}{\gamma+1}(D_\e))} \\
&\leq \|\sqrt{\rho_\e}\|_{L^\infty(0,T;L^{2\gamma}(D_\e))} \|\sqrt{\rho_\e}\vu_\e\|_{L^\infty(0,T;L^2(D_\e))} \lesssim 1.
\end{align*}

\subsection{Improved integrability of the density}
The uniform bound for the density is not enough to pass to the limit in the pressure since we just have $p(\rho_\e)$ uniformly bounded in $L^\infty(0,T;L^1(D_\e))$. To pass to a suitable weakly convergent subsequence, we need the following
\begin{lemma}\label{lem:RegDens}
Let $\theta = \frac23 \gamma - 1$. Then
\begin{align*}
\int_0^T \int_{D_\e} \rho_\e^{\gamma+\theta} \lesssim 1.
\end{align*}
\end{lemma}
Before we prove Lemma~\ref{lem:RegDens}, let us recall two known facts. First, a crucial tool for the sequel will be the Bogovski\u{\i} operator constructed in \cite[Theorem~2.3]{DieningFeireislLu2017}.
\begin{theorem}\label{thm:Bog}
Let $1 < q < \infty$. There exists a bounded linear operator
\begin{align*}
\B_\e: L_0^q(D_\e) = \Big\{ f\in L^q(D_\e) : \int_{D_\e} f \dd x = 0 \Big\} \to W_0^{1,q}(D_\e)
\end{align*}
such that for any $f\in L_0^q(D_\e)$,
\begin{align*}
\div \B_\e(f) = f,\quad \|\B_\e(f)\|_{W_0^{1,q}(D_\e)}^q \lesssim \Big(1+ \e^{(3-q)\alpha-3}\Big) \|f\|_{L^q(D_\e)}^q.
\end{align*}
\end{theorem}
We remark that a similar operator with the same scaling in $\e$ for any $\frac32 < q < 3$ was given in \cite{Lu2021}, the construction of which relies on the construction of a so-called restriction operator. We will come back to this in Section~\ref{sec32}.\\

Second, we recall a result from \cite[Proposition~2.2]{LuSchwarzacher2018}, which we state in form of a lemma.
\begin{lemma}
Let $\B_\e$ be the Bogovski\u{\i} operator from Theorem~\ref{thm:Bog}. Then, for any $r>\frac32$, we can extend $\B_\e$ to an operator
\begin{align*}
\B_\e: \Big\{ g=\div \vc f \in [W^{-1, r'}(D_\e)]' : \langle g, 1\rangle=0 \Big\} \to L^r(D_\e)
\end{align*}
such that
\begin{align*}
\langle \B_\e \div \vc f, \nabla \phi\rangle = \langle \vc f, \nabla \phi\rangle \ \text{for any} \ \phi\in W^{1,r'}(D_\e), \quad \|\B_\e \div \vc f\|_{L^r(D_\e)} \lesssim \|\vc f\|_{L^r(D_\e)}.
\end{align*}
\end{lemma}

We are now in the position to prove Lemma~\ref{lem:RegDens}.
\begin{proof}[Proof of Lemma~\ref{lem:RegDens}]
We remark that this Lemma was proven for $\gamma>6$ in \cite{LuSchwarzacher2018}. However, we need to deal with our value $\gamma>3$, so we recall the proof, and also give an easier argument how to deal with the term containing the time-derivative of the Bogovski\u{\i} operator. The idea is to test the momentum equation by
\begin{align*}
\phi(t,x) = \psi(t) \B_\e \Big(\rho_\e^\theta - \frac{1}{|D_\e|} \int_{D_\e} \rho_\e^\theta \dd x \Big), \quad \theta = \frac23 \gamma - 1,
\end{align*}
for some $\psi \in C_c^\infty([0,T))$. The proof then follows the same lines as \cite[pp.~77-82]{OschmannDiss2022}, once observed that $\gamma>3$ is enough to repeat the steps done there. For the sake of completeness, we will present the main steps of the proof here. Taking $\theta= \frac 23 \gamma-1$, $\gamma >3$, and $\psi(t)=1$ for $0\leq t\leq T-\delta$ for some small $\delta >0$, $\psi (t) \geq 0$ in $[0,T]$, we have
\begin{align}\label{eq1}
\begin{split}
& \int_0^T \int_{D_\e} \psi \rho_\e^{\gamma + \theta} \dd x \dd t = \int_0^T \psi \bigg( \int_{D_\e} \rho_\e^\gamma \dd x \bigg) \bigg( \frac{1}{|D_\e|} \int_{D_\e} \rho_\e^{\theta} \dd x \bigg) \dd t  - \int_{D_\e} \vc m_0 \cdot \B_\e^* \dd x \\
& - \int_0^T \int_{D_\e} \psi \rho_\e \vc f \cdot \B_\e^* \dd x \dd t
-\int_0^T \int_{D_\e} \psi' \rho_\e \vu_\e \cdot \B_\e^* \dd x \dd t +  \int_0^T \int_{D_\e} \psi \ms(\nabla \vu_\e):\nabla \B_\e^* \dd x \dd t \\
& - \int_0^T \int_{D_\e} \psi \rho_\e \vu_\e\otimes \vu_\e : \nabla \B_\e^* \dd x \dd t -\int_0^T \int_{D_\e} \psi \rho_\e \vu_\e \cdot \partial_t \B_\e^* \dd x \dd t  = \sum_{i=1}^7 I_i.
\end{split}
\end{align}
Above, 
\begin{align*}
\B_\e^* = \B_\e \Big(\rho_\e^{\theta} - \frac{1}{|D_\e|} \int_{D_\e} \rho_\e^{\theta}\Big).
\end{align*}
Clearly, the most restrictive terms are the three last ones and we will only concentrate on estimates of them. First, note that due to \eqref{cond1} we always have $\alpha>3$. It is easy to see that
\begin{align*}
&|I_5| \lesssim   \|\nabla \vu_\e\|_{L^2((0,T)\times D_\e)} \|\sqrt{\psi} \nabla \B_\e^*\|_{L^2((0,T)\times D_\e)} \\
&\lesssim (1+ \e^{\alpha -3})^\frac12 \Big(\int_0^T \int_{D_\e} \psi \rho_\e^{2\theta} \dd x \dd t\Big)^{\frac 12} 
 \leq C + \frac 14 \int_0^T \int_{D_\e} \psi \rho_\e^{\gamma + \theta} \dd x \dd t    
\end{align*} 
and the last term can be transferred to the left-hand side of \eqref{eq1}. Next,
\begin{align*}
&|I_6| \lesssim \|\rho_\e\|_{L^\infty(0,T; L^\gamma(D_\e))} \|\vu_\e\|_{L^2(0,T;L^6(D_\e))}^2 \|\nabla \B_\e^*\|_{L^\infty(0,T; L^{\frac{3\gamma}{2\gamma-3}}(D_\e))} \\
&\lesssim \big(1+ \e^{(3-\frac{3\gamma}{2\gamma-3})\alpha -3}\big)^\frac{2\gamma-3}{3\gamma} \|\rho_\e\|_{L^\infty(0,T; L^\gamma(D_\e))}^{\theta} \|\vu_\e\|_{L^2(0,T;L^6(D_\e))}^2 \lesssim \big(1+ \e^{ 3\alpha\frac{\gamma-3}{2\gamma-3}-3}\big)^\frac{2\gamma-3}{3\gamma} \lesssim 1
\end{align*}
due to $\frac{3\gamma}{2\gamma-3} \in (\tfrac 32,3)$ and \eqref{cond1}. The last term is more complex. Using the renormalized continuity equation \eqref{renCE} with $b(s)=s^\theta$, i.e.,
\begin{align*}
\partial_t \rho_\e^\theta + \div(\rho_\e^\theta \vu_\e) + (\theta-1)\rho_\e^\theta \div \vu_\e = 0 \qquad \text{ in } \mathcal D'((0,T)\times D_\e),
\end{align*}
we get that
\begin{align*} 
&I_7 = \int_0^T \int_{D_\e} \psi \rho_\e \vu_\e \cdot \B_\e(\div(\rho_\e^{\theta} \vu_\e)) \dd x \dd t  \\
& + (\theta-1) \int_0^T \int_{D_\e} \psi \rho_\e \vu_\e \cdot \B_\e\Big(\rho_\e^{\theta}\div \vu_\e- \frac{1}{|D_\e|} \int_{D_\e} \rho_\e^{\theta}\div\vu_\e \dd x \Big) \dd x \dd t  = I_{7,1} + I_{7,2}.
\end{align*}
In order to estimate the term $I_{7,1}$, we choose the space and time exponent in the estimate of the Bogovski\u{\i} operator in such a way that the density coming from the Bogovski\u{\i} operator will be estimated precisely in the $L^{\gamma + \theta}$-norm over the space-time so that this term will be finally absorbed into the left-hand side of \eqref{eq1}. Recalling the estimate of the linear momentum
\begin{align*}
\|\rho_\e\vu_\e\|_{L^\infty(0,T;L^{\frac{2\gamma}{\gamma+1}}(D_\e))} \lesssim 1,
\end{align*}
and, in view of the bound
\begin{align*}
\|\rho_\e\|_{L^\infty(0,T;L^\gamma(D_\e))} + \|\vu_\e\|_{L^2(0,T;L^6(D_\e))} \lesssim 1,
\end{align*}
we also have
\begin{align*}
\|\rho_\e \vu_\e\|_{L^2(0,T;L^\frac{6\gamma}{\gamma+6}(D_\e))} \lesssim 1.
\end{align*}
We use interpolation between Lebesgue spaces to obtain
\begin{align*}
\|\rho_\e \vu_\e\|_{L^\frac{2(5\gamma-3)}{\gamma+3}(0,T;L^\frac{6(5\gamma-3)}{13\gamma+3}(D_\e))} \leq \|\rho_\e \vu_\e\|_{L^\infty(0,T;L^\frac{2\gamma}{\gamma+1}(D_\e))}^{1-\varpi} \|\rho_\e \vu_\e\|_{L^2(0,T;L^\frac{6\gamma}{\gamma+6}(D_\e))}^\varpi \lesssim 1,
\end{align*}
where $\varpi = \frac{\gamma+3}{5\gamma-3} \in (\frac15, \frac12)$. Hence, by $\frac{6(5\gamma-3)}{17\gamma-21} > \frac 32$ for any $\gamma>\frac{21}{17}$, we get
\begin{align*}
&|I_{7,1}| \lesssim \|\rho_\e \vu_\e\|_{L^{\frac{2(5\gamma-3)}{\gamma+3}}(0,T; L^{\frac{6(5\gamma-3)}{13\gamma +3}}(D_\e))} \|\psi^\frac{\theta}{\gamma+\theta} \B_\e(\div(\rho_\e^{\theta} \vu_\e))\|_{L^{\frac{2(5\gamma-3)}{9(\gamma-1)}}(0,T;L^{\frac{6(5\gamma-3)}{17\gamma-21}}(D_\e))} \\
&\lesssim \|\psi^\frac{\theta}{\gamma+\theta} \rho_\e^{\theta} \vu_\e\|_{L^{\frac{2(5\gamma-3)}{9(\gamma-1)}}(0,T;L^{\frac{6(5\gamma-3)}{17\gamma-21}}(D_\e))}.
\end{align*}
Further, note that by $\theta = \frac23 \gamma - 1$,
\begin{align*}
\frac{9(\gamma-1)}{2(5\gamma-3)} = \frac{\theta}{\gamma+\theta} + \frac12,\qquad \frac{17\gamma-21}{6(5\gamma-3)} = \frac{\theta}{\gamma+\theta} + \frac16,
\end{align*}
so we finally arrive at
\begin{align*}
|I_{7,1}| &\lesssim \|\psi^\frac{\theta}{\gamma+\theta} \rho_\e^{\theta} \vu_\e\|_{L^{\frac{2(5\gamma-3)}{9(\gamma-1)}}(0,T;L^{\frac{6(5\gamma-3)}{17\gamma-21}}(D_\e))} \lesssim \|\psi^\frac{\theta}{\gamma+\theta} \rho_\e^\theta\|_{L^\frac{\gamma+\theta}{\theta}((0,T)\times D_\e)} \|\vu_\e\|_{L^2(0,T;L^6(D_\e))} \\
&= \|\psi^\frac{1}{\gamma+\theta}\rho_\e\|_{L^{\gamma+\theta}((0,T)\times D_\e)}^\theta \|\vu_\e\|_{L^2(0,T;L^6(D_\e))} \lesssim 1 + \frac14 \int_0^T \int_{D_\e} \psi \rho_\e^{\gamma+\theta} \dd x \dd t.
\end{align*}
The second term on the right-hand side can be absorbed by the left-hand side of equation \eqref{eq1}. Similarly we proceed in the estimate of $I_{7,2}$. Abbreviating $\langle \rho_\e^\theta \div \vu_\e \rangle = \frac{1}{|D_\e|} \int_{D_\e} \rho^\theta \div \vu_\e \dd x$, by the Sobolev embedding $W_0^{1,\frac{2(5\gamma-3)}{9(\gamma-1)}} \subset L^\frac{6(5\gamma-3)}{17\gamma-21}$, we have

\begin{align*}
|I_{7,2}| &\lesssim \|\rho_\e \vu_\e\|_{L^{\frac{2(5\gamma-3)}{\gamma+3}}(0,T; L^{\frac{6(5\gamma-3)}{13\gamma +3}}(D_\e))} \|\psi^\frac{\theta}{\gamma+\theta} \B_\e(\rho_\e^{\theta} \div\vu_\e - \langle \rho_\e^\theta \div \vu_\e \rangle)\|_{L^{\frac{2(5\gamma-3)}{9(\gamma-1)}}(0,T;L^{\frac{6(5\gamma-3)}{17\gamma-21}}(D_\e))} \\
&\lesssim \|\psi^\frac{\theta}{\gamma+\theta} \nabla \B_\e(\rho_\e^{\theta} \div\vu_\e - \langle \rho_\e^\theta \div \vu_\e \rangle)\|_{L^{\frac{2(5\gamma-3)}{9(\gamma-1)}}(0,T;L^{\frac{2(5\gamma-3)}{9(\gamma-1)}}(D_\e))} \\
&\lesssim \big( 1 + \e^{(3-\frac{2(5\gamma-3)}{9(\gamma-1)})\alpha-3} \big)^\frac{9(\gamma-1)}{2(5\gamma-3)} \|\psi^\frac{\theta}{\gamma+\theta} \rho_\e^\theta \div \vu_\e\|_{L^\frac{2(5\gamma-3)}{9(\gamma-1)}((0,T)\times D_\e)}\\
&\lesssim \|\psi^\frac{\theta}{\gamma+\theta} \rho_\e^\theta\|_{L^\frac{\gamma+\theta}{\theta}((0,T)\times D_\e)} \|\div \vu_\e\|_{L^2((0,T)\times D_\e)}
 \lesssim 1 + \frac 14 \int_0^T \int_{D_\e} \psi \rho_\e^{\gamma + \theta}\dd x \dd t. 
\end{align*}
Note that the exponent of $\e$ is non-negative due to $3-\frac{2(5\gamma-3)}{9(\gamma-1)}=\frac{17\gamma-21}{9(\gamma-1)} > 1$ for any $\gamma>\frac32$. Again, the last term on the right-hand side can be absorbed into the left-hand side of \eqref{eq1}. The lemma is proved.
\end{proof}

\subsection{Refined pressure estimates}\label{sec32}
A key tool in our analysis is a suitable extension of the pressure and corresponding bounds. To this end, we recall the following theorem from \cite[Theorem~2.1]{Lu2021}.
\begin{theorem}\label{thm:Restr}
For any $1 < q < \infty$, there exists a bounded linear operator
\begin{align*}
R_\e: W_0^{1,q}(D) \to W_0^{1,q}(D_\e)
\end{align*}
such that
\begin{align*}
&R_\e \tvu = \vu \ \text{for any} \ \vu \in W_0^{1,q}(D_\e),\\
&\div \vu = 0 \Rightarrow \div R_\e \vu = 0.
\end{align*}
Moreover, for any $\frac32 < q < 3$,
\begin{align*}
\|R_\e \vu\|_{W_0^{1,q}(D_\e)}^q \lesssim \Big(1+ \e^{(3-q)\alpha-3}\Big) \|\nabla \vu\|_{L^q(D)}^q.
\end{align*}
\end{theorem}

\begin{rem}
As a matter of fact, the operator $\B_\e$ from Theorem~\ref{thm:Bog} also relies on the construction of a similar, but different restriction operator $\widehat{R}_\e$ in the sense that $\B_\e(f) = \widehat{R}_\e (\B_D(\tilde{f}))$ for any $f\in L_0^q(D_\e)$, where $\B_D$ is the Bogovski\u{\i} operator on $D$ satisfying
\begin{align*}
\B_D: L_0^q(D) \to W_0^{1,q}(D),\quad \div \B_D(f) = f, \quad \|\B_D(f)\|_{W_0^{1,q}(D)} \lesssim \|f\|_{L^q(D)},
\end{align*}
see \cite{Bogovskii1980, Galdi2011}. The important point is that the operator $\widehat{R}_\e$ does not satisfy $\widehat{R}_\e \tvu = \vu$ for any $\vu \in W_0^{1,q}(D_\e)$, which will be a crucial property in our analysis below.
\end{rem}

We now use the operator $R_\e$ from Theorem~\ref{thm:Restr} to define similarly to \cite{Allaire1990a} and \cite{Masmoudi2002} a pressure extension $P_\e$ as
\begin{align*}
\int_0^T \langle \nabla P_\e, \phi\rangle_{\Test', \Test(D)} \dd t = \int_0^T \langle \nabla p(\rho_\e), R_\e \phi\rangle_{\Test', \Test(D_\e)} \dd t \ \text{for any} \ \phi \in C_c^\infty([0,T) \times D).
\end{align*}
It follows from \cite[Proposition~1.1.4]{Allaire1990a} that $P_\e = p(\rho_\e)$ on $D_\e$, up to an additive constant. Then, using the second equation of \eqref{NSE}, we deduce
\begin{align*}
\int_0^T \langle \nabla p(\rho_\e) , R_\e \phi\rangle_{\Test', \Test(D_\e)} \dd t &= \int_0^T \int_{D_\e} \rho_\e \vu_\e \cdot \del_t R_\e \phi \dd x \dd t + \int_0^T \int_{D_\e} \rho_\e \vu_\e \otimes \vu_\e : \nabla R_\e \phi \dd x \dd t \\
&\quad - \int_0^T\int_{D_\e} \ms(\nabla \vu_\e):\nabla R_\e \phi \dd x \dd t + \int_0^T\int_{D_\e} \rho_\e \vc f \cdot R_\e \phi \dd x \dd t\\
&\quad + \int_{D_\e} \vc m_{\e,0} \cdot R_\e \phi(0,\cdot) \dd x\\
&= \sum_{j=1}^5 I_j.
\end{align*}

We estimate each term separately and start with $I_2$, which is the most restrictive one. Recalling the Sobolev embedding $W_0^{1,2} \subset L^6$ and the uniform bounds \eqref{unifbds}, we get
\begin{align*}
I_2 &\lesssim \|\rho_\e |\vu_\e|^2\|_{L^1(0,T; L^\frac{3\gamma}{\gamma+3}(D_\e))} \|\nabla R_\e \phi\|_{L^\infty(0,T; L^\frac{3\gamma}{2\gamma-3}(D_\e))} \\
&\lesssim \|\rho_\e\|_{L^\infty(0,T; L^\gamma(D_\e))} \|\vu_\e\|_{L^2(0,T; L^6(D_\e))}^2 \Big( 1+ \e^{(3 - \frac{3\gamma}{2\gamma-3})\alpha - 3} \Big)^\frac{2\gamma-3}{3\gamma} \|\nabla \phi\|_{L^\infty(0,T; L^\frac{3\gamma}{2\gamma-3}(D))}.
\end{align*}
Note that the bound for $I_2$ is uniform since
\begin{align}\label{jedna}
3-\frac{3\gamma}{2\gamma-3}=\frac{3(\gamma-3)}{2\gamma-3},
\end{align}
which is governed precisely by condition \eqref{cond1}.

For $I_1$, we apply the Sobolev embedding $W_0^{1,\frac{6\gamma}{5\gamma-3}} \subset L^\frac{2\gamma}{\gamma-1}$ as well as $\frac{6\gamma}{5\gamma-3}<2$ for any $\gamma>\frac32$ to get
\begin{align*}
I_1 &\lesssim \|\rho_\e \vu_\e\|_{L^2(0,T; L^\frac{2\gamma}{\gamma+1}(D_\e))} \|\del_t R_\e \phi\|_{L^2(0,T; L^\frac{2\gamma}{\gamma-1}(D_\e))} \lesssim \|\del_t R_\e \phi\|_{L^2(0,T; W_0^{1,\frac{6\gamma}{5\gamma-3}}(D_\e))}\\
& \lesssim \|\del_t R_\e \phi\|_{L^2(0,T; W_0^{1,2}(D_\e))} \lesssim \big( 1 + \e^{\alpha - 3}\big)^\frac12 \|\del_t \nabla \phi\|_{L^2(0,T; L^2(D))}\\
&\lesssim \|\del_t \nabla \phi\|_{L^2((0,T) \times D)},
\end{align*}
where we used that $\alpha>3$ by \eqref{cond1}.
For $I_5$, we use the fact that $\phi(T)=0$ to write
\begin{align*}
\phi(\tau) = -\int_\tau^T \del_t \phi \dd t \Rightarrow \|\phi\|_{L^\infty(0,T)} \leq \int_0^T |\del_t \phi| \dd t \lesssim \|\del_t \phi\|_{L^2(0,T)}
\end{align*}
and estimate similar to $I_1$
\begin{align*}
I_5 &\lesssim \|\vc m_{\e, 0}\|_{L^\frac{2\gamma}{\gamma+1}(D_\e)} \|R_\e \phi(0)\|_{L^\frac{2\gamma}{\gamma-1}(D_\e)} \lesssim \|R_\e \phi(0)\|_{W_0^{1,\frac{6\gamma}{5\gamma-3}}(D_\e)} \\
& \lesssim \big( 1 + \e^{\alpha - 3}\big)^\frac12 \|R_\e \phi(0)\|_{W_0^{1,2}(D_\e)} \lesssim \|\nabla \phi(0)\|_{L^2(D)}\\
& \lesssim \|\nabla \phi\|_{L^\infty(0,T;L^2(D))} \lesssim \|\del_t \nabla \phi\|_{L^2((0,T) \times D)},
\end{align*}
where we also used that due to the initial conditions \eqref{initConv}, we have
\begin{align*}
\|\vc m_{\e, 0}\|_{L^\frac{2\gamma}{\gamma+1}(D_\e)} \leq \left\|\frac{\tilde{\vc m}_{\e, 0}}{\sqrt{\tvr_{\e, 0}}}\right\|_{L^2(D)} \left\|\sqrt{\tvr_{\e, 0}} \right\|_{L^{2\gamma}(D)} \lesssim 1.
\end{align*}

For $I_3$ and $I_4$, we use the assumption $\alpha > 3$ to estimate
\begin{align*}
I_3 + I_4 &\lesssim \Big( \|\ms(\nabla \vu_\e)\|_{L^2(0,T; L^2(D_\e))} + \|\rho_\e\|_{L^2(0,T; L^2(D_\e))} \Big) \|R_\e \phi\|_{L^2(0,T; W_0^{1,2}(D_\e))}\\
&\lesssim \big (1+ \e^{\alpha-3} \big)^\frac12 \|\nabla \phi\|_{L^2(0,T; L^2(D))} \lesssim \|\nabla \phi\|_{L^2(0,T; L^2(D))}.
\end{align*}

Altogether, we have
\begin{align*}
\nabla P_\e \ \text{uniformly bounded in} \ &L^1(0,T; W^{-1, \frac{3\gamma}{\gamma+3}}(D))  + \widehat{W}^{-1,2}(0,T; W^{-1, 2}(D))\\
& + L^2(0,T; W^{-1,2}(D)),
\end{align*}
where $\widehat{W}^{-1,2}(0,T)$ is the dual space to $(W^{1,2}(0,T) \cap \{\phi(T)=0\}, \|\del_t \cdot\|_{L^2(0,T)})$.  This implies
\begin{align*}
P_\e \ \text{uniformly bounded in} \ L^1(0,T; L^\frac{3\gamma}{\gamma+3}(D)) + \widehat{W}^{-1,2}(0,T; L^2(D)) + L^2(0,T; L^2(D)),
\end{align*}
which we can simplify to
\begin{align}\label{pressDec}
P_\e \ \text{uniformly bounded in} \ L^1(0,T;L^{\min \{2, \frac{3\gamma}{\gamma+3} \}}(D)) + \widehat{W}^{-1,2}(0,T;L^2(D)).
\end{align}

\begin{rem}\label{rem1}
We remark that $\frac{3\gamma}{\gamma+3}>\frac32$ for any $\gamma>3$, which coincides with the observation made in \cite[Section~1.2.2]{LuSchwarzacher2018} that the pressure should be at least in $L^\frac32$ in space, leading to $\gamma>6$ for the pure bound $p(\rho_\e) \in L^{\frac53 - \frac1\gamma}((0,T) \times D_\e)$ obtained in Lemma~\ref{lem:RegDens}.
\end{rem}

\begin{rem}\label{rem2}
Following the proof of \cite[Propositions~1.1.4 and 2.1.2]{Allaire1990a} and defining $C_i^\e = B_\e(x_i(\e)) \setminus B_{\e^\alpha}(x_i(\e))$, we find that $P_\e$ has the explicit representation
\begin{align*}
P_\e = \begin{cases}
p(\rho_\e) & \text{in } (0,T) \times D_\e,\\
\frac{1}{|C_i^\e|} \int_{C_i^\e} p(\rho_\e) \dd x & \text{in } (0,T) \times B_{\e^\alpha}(x_i(\e)), \ i=1,...,N(\e).
\end{cases}
\end{align*}
From this and Lemma~\ref{lem:RegDens}, we infer that
\begin{align*}
P_\e \text{ uniformly bounded in } L^\infty(0,T;L^1(D)) \cap L^{\frac53 - \frac1\gamma}((0,T) \times D).
\end{align*}
{\cre In particular, the part of $P_\e$ belonging to $\widehat{W}^{-1,2}(0,T;L^2(D))$ is actually a function rather than a mere distribution.}
\end{rem}

\section{Convergences}\label{sec:Conv}
\subsection{Limiting functions}
As a consequence of the uniform bounds given in Section~\ref{sec:Bds}, we obtain, at least for a subsequence,
\begin{align*}
\tvu_\e &\weak \vu \ \text{weakly in} \ L^2(0,T;W_0^{1,2}(D)),\\
\tvr_\e &\weak^\ast \rho \ \text{weakly-$\ast$ in} \ L^\infty(0,T;L^\gamma(D)) \ \text{and weakly in} \ L^{\frac53 \gamma - 1}((0,T) \times D),\\
p(\tvr_\e) &\weak \overline{p(\rho)} \ \text{weakly in} \ L^{\frac53-\frac1\gamma}((0,T)\times D).
\end{align*}

\subsection{Limit in the continuity equation}
Repeating the arguments given in \cite{LuSchwarzacher2018, PokornySkrisovski2021}, we recover
\begin{align}\label{approx_cont}
&\del_t \tvr_\e + \div (\tvr_\e \tvu_\e) = 0 \ \text{in} \ \Test'((0,T) \times \R^3)
\end{align}
as well as, after the limit passage,
\begin{align*}
\del_t \tvr + \div (\tvr \tvu) = 0 \ \text{in} \ \Test'((0,T) \times \R^3).
\end{align*}
Indeed, from the uniform bounds on $\tilde{\rho}_\e$ and $\tilde{\vc u}_\e$ derived in Section~\ref{sec:Bds}, we have
\begin{align*}
\|\tilde{\rho}_\e\tilde{\vc u}_\e\|_{L^\infty(0,T;L^\frac{2\gamma}{\gamma+1}(D))}\leq \|\sqrt{\tilde{\rho}_\e}\|_{L^\infty(0,T;L^{2\gamma}(D))} \|\sqrt{\tilde{\rho}_\e}\tilde{\vc u}_\e\|_{L^\infty(0,T;L^2(D))}\leq C.
\end{align*}
Moreover, by \eqref{approx_cont}, we have
\begin{align*}
\del_t \tilde{\rho}_\e \text{ bounded in } L^2(0,T;W^{-1,p}(D)) \text{ for some } p>1.
\end{align*}
Applying \cite[Lemma~5.1]{Lions1998} now shows
\begin{align*}
\tilde{\rho}_\e\tilde{\vc u}_\e \to \rho\vc u \text{ in } \mathcal{D}'((0,T)\times D).
\end{align*}
Furthermore, an Aubin-Lions type argument yields
\begin{align}\label{AL}
\tilde{\rho}_\e \to \rho \text{ in } C(0,T;L_{\rm weak}^\gamma(D)),\quad \tilde{\rho}_\e \tilde{\vc u}_\e \to \rho \vc u \text{ in } C(0,T;L_{\rm weak}^\frac{2\gamma}{\gamma+1}(D)).
\end{align}
A similar argument applies to $\tilde{\rho}_\e\tilde{\vc u}_\e\otimes\tilde{\vc u}_\e$, which is needed in the limit passage in the momentum equation below.
As a consequence of \cite[Lemma~6.9]{Novotny2004}, the couples $(\tvr_\e, \tvu_\e)$ as well as $(\tvr, \tvu)$ also fulfil the renormalized continuity equation in $\Test'((0,T) \times \R^3)$, that is, the equations given in \eqref{renCE}.

\subsection{Limit in the momentum equation}
Before passing with $\e\to 0$ in the momentum equation, we state the following lemma, which is an immediate consequence of \eqref{defKeps}.
\begin{lemma}\label{lem:Cutoff}
There exist functions $g_\e \in C^\infty(D)$ such that
\begin{align*}
g_\e = 0 \ \text{on} \ \bigcup_{i=1}^{N(\e)} B_{\e^\alpha}(x_i(\e)),\quad g_\e = 1 \ \text{on} \ \bigcup_{i=1}^{N(\e)} B_{2 \e^\alpha}(x_i(\e)),
\end{align*}
and satisfying
\begin{align*}
\|1-g_\e\|_{L^q(D)}^q\lesssim \e^{3(\alpha-1)},\quad \|\nabla g_\e\|_{L^q(D)}^q\lesssim \e^{(3-q)\alpha - 3}.
\end{align*}
\end{lemma}
For $\phi\in C_c^\infty([0,T)\times D)$, we test the second equation of \eqref{NSE} by $\phi = g_\e \phi + (1-g_\e) \phi$, where $g_\e$ are as in Lemma~\ref{lem:Cutoff}. Using $P_\e = p(\rho_\e)$ on $D_\e$ and seeing that $g_\e \phi$ is a proper test function, we obtain 
\begin{align} \label{approx_mom}
\begin{split}
&\int_0^T\int_{D} \tvr_\e \tvu_\e \cdot \del_t \phi \dd x \dd t + \int_0^T\int_{D} \tvr_\e \tvu_\e\otimes \tvu_\e : \nabla \phi \dd x \dd t + \int_0^T\int_{D} P_\e\div \phi \dd x \dd t \\
&- \int_0^T\int_{D} \ms(\nabla \tvu_\e):\nabla \phi \dd x \dd t + \int_0^T\int_{D} \tvr_\e \vc f \cdot \phi \dd x \dd t +\int_{D} \tilde{\vc m}_{\e,0} \cdot \phi(0,\cdot) \dd x\\
&= \int_0^T\int_{D_\e} \rho_\e \vu_\e \cdot \del_t (g_\e \phi) \dd x \dd t + \int_0^T\int_{D_\e} \rho_\e \vu_\e\otimes \vu_\e : \nabla (g_\e \phi) \dd x \dd t + \int_0^T\int_{D_\e} p(\rho_\e) \div (g_\e \phi) \dd x \dd t \\
&- \int_0^T\int_{D_\e} \ms(\nabla \vu_\e):\nabla (g_\e \phi) \dd x \dd t + \int_0^T\int_{D_\e} \rho_\e \vc f \cdot (g_\e \phi) \dd x \dd t +\int_{D_\e} \vc m_{\e,0} \cdot (g_\e \phi(0,\cdot)) \dd x + F_\e\\
&= F_\e
\end{split}
\end{align}
with the remainder
\begin{align*}
F_\e &= \int_0^T\int_D \tvr_\e \tvu_\e \cdot \del_t ((1-g_\e)\phi) \dd x \dd t + \int_0^T \int_D \tvr_\e \tvu_\e \otimes \tvu_\e : \nabla ((1-g_\e)\phi) \dd x \dd t\\
&\quad + \int_0^T \int_D P_\e \div((1-g_\e)\phi) \dd x \dd t - \int_0^T \int_D \ms(\nabla \tvu_\e) : \nabla ((1-g_\e)\phi) \dd x \dd t\\
&\quad + \int_0^T \int_D \tvr_\e \vc f \cdot (1-g_\e) \phi \dd x \dd t + \int_D \tilde{\vc m}_{\e, 0} \cdot (1-g_\e)\phi(0,\cdot) \dd x.
\end{align*}

In order to prove Theorem~\ref{thm1}, we have to show first that $F_\e \to 0$ as $\e \to 0$. Let us start with the most crucial terms. For the pressure $P_\e$, according to \eqref{pressDec} we split $P_\e = P_{\e, 1} + P_{\e, 2}$ with
\begin{gather*}
P_{\e, 1} \ \text{uniformly bounded in} \ L^1(0,T;L^{\min\{2, \frac{3\gamma}{\gamma+3} \} }(D)),\\
P_{\e, 2} \ \text{uniformly bounded in} \ \widehat{W}^{-1,2}(0,T;L^2(D)).
\end{gather*}
Then, we estimate for $P_{\e, 1}$
\begin{align*}
&\Big| \int_0^T \int_D P_{\e, 1} \div((1-g_\e)\phi) \dd x \dd t \Big| \leq \Big|\int_0^T \int_D P_{\e, 1} (1-g_\e) \div \phi \dd x \dd t \Big| \\ & + \Big|\int_0^T \int_D P_{\e, 1} \nabla g_\e \cdot \phi \dd x \dd t \Big|\\
&\lesssim \|P_{\e, 1}\|_{L^1(0,T;L^{\min\{2, \frac{3\gamma}{\gamma+3} \} }(D))} \big( \|1-g_\e\|_{L^{\max\{2, \frac{3\gamma}{2\gamma-3} \} }(D)} + \|\nabla g_\e\|_{L^{\max\{2, \frac{3\gamma}{2\gamma-3} \} }(D)} \big) \|\phi\|_{W^{1,\infty}((0,T) \times D)}\\
&\lesssim \e^{3(\alpha-1)\min\{\frac12, \frac{2\gamma-3}{3\gamma} \} } + \e^{\big((3-\max\{2, \frac{3\gamma}{2\gamma-3} \})\alpha - 3 \big)\min\{\frac12, \frac{2\gamma-3}{3\gamma} \}} \to 0
\end{align*}
by \eqref{jedna} and assumption \eqref{cond1}.

Similarly,
\begin{align*}
&\Big| \int_0^T \int_D P_{\e, 2} \div((1-g_\e)\phi) \dd x \dd t \Big| \leq \Big|\int_0^T \int_D P_{\e, 2} (1-g_\e) \div \phi \dd x \dd t \Big| \\
&+ \Big|\int_0^T \int_D P_{\e, 2} \nabla g_\e \cdot \phi \dd x \dd t \Big|\\
&\lesssim \|P_{\e, 2}\|_{\widehat{W}^{-1,2}(0,T;L^2(D))} \big( \|1-g_\e\|_{L^2(D)} + \|\nabla g_\e\|_{L^2(D)} \big) \|\del_t \phi\|_{L^2(0,T;W^{1,\infty}(D))}\\
&\lesssim \e^{\frac{3(\alpha-1)}{2}} + \e^{\frac{\alpha-3}{2}} \to 0.
\end{align*}

Seeing that the estimates for the remaining terms of $F_\e$ are the same as the ones obtained in Section~\ref{sec32}, where every $R_\e \phi$ has to be replaced by $(1-g_\e)\phi$, and any $1 + \e^{(3-q)\alpha - 3}$ by $\e^{(3-q)\alpha - 3}$, we indeed see that $F_\e \to 0$. It remains to show that $P_\e \to \rho^\gamma$ at least in $L^1_{\rm loc}((0,T)\times D)$. In order to show this we may apply the standard procedure based on the effective viscous flux identity combined with the renormalized continuity equation. The first step is connected with the use of the test function 
\begin{align*}
\phi(t,x)= \psi(t) \varphi(x) \nabla\Delta^{-1} \tvr_\e^\beta
\end{align*}
in \eqref{approx_mom}
and 
\begin{align*}
\phi(t,x)= \psi(t) \varphi(x) \nabla\Delta^{-1}\overline{\rho^\beta}
\end{align*}
in the limit form of the momentum equation (the limit passage in the time derivative and the convective term is well known, see, e.g., \cite{Novotny2004})
\begin{align} \label{limit_mom}
\begin{split}
&\int_0^T\int_{D} \rho \vu \cdot \del_t \phi \dd x \dd t + \int_0^T\int_{D} \rho \vu\otimes \vu : \nabla \phi \dd x \dd t + \int_0^T \int_{D} \overline{\rho^\gamma} \div \phi \dd x \dd t \\
&- \int_0^T\int_{D} \ms(\nabla \vu):\nabla \phi \dd x \dd t + \int_0^T\int_{D} \rho \vc f \cdot \phi \dd x \dd t = 0. 
\end{split}
\end{align}                      
In both situations, $\psi \in C^\infty_c(0,T)$, $\varphi \in C^\infty_c(D_\e)$, and $\beta >0$ sufficiently small so that all integrals are finite. The only slight technical point is connected  with the question whether we may replace in the momentum equation the term $P_\e$ by $p(\tvr_\e)$ which leads to the question whether
\begin{align*}
\int_0^T \int_D (P_\e-\tvr_\e^\gamma)\div \phi \dd x \dd t \to 0 \qquad \text{ for } \phi \in C^\infty_c((0,T)\times D)
\end{align*}
as well as   
\begin{align*}
\int_0^T \int_D (P_\e-\tvr_\e^\gamma)\div (\psi (t)\varphi(x)\nabla\Delta^{-1} \tvr_\e^\beta) \dd x \dd t \to 0 
\end{align*}
with $\psi(t)$ and $\varphi(x)$ as above. Indeed, since $P_\e = p(\rho_\e) = \rho_\e^\gamma$ in $(0,T)\times D_\e$  and $\tvr_\e = 0$ in $D\setminus D_\e$, for the first term we have
\begin{align*}
\int_0^T \int_D (P_\e-\tvr_\e^\gamma)\div \phi \dd x \dd t = \int_0^T \int_{D\setminus D_\e} P_\e \div \phi \dd x \dd t \lesssim \|P_\e\|_{L^{\frac53 - \frac1\gamma}((0,T) \times D)} |D\setminus D_\e|^\frac{2\gamma-3}{5\gamma-3} \to 0.
\end{align*}
The second question is more difficult. However, again by $P_\e = \rho_\e^\gamma$ in $(0,T)\times D_\e$, we have
\begin{align*}
&\int_0^T \int_D (P_\e-\tvr_\e^\gamma)\div (\psi(t) \varphi(x)\nabla\Delta^{-1} \tvr_\e^\beta) \dd x \dd t \\
&= \int_0^T \int_{D\setminus D_\e} P_\e\psi(t) \varphi(x) \tvr_\e^\beta \dd x \dd t + \int_0^T \int_{D\setminus D_\e} P_\e \psi(t) \nabla \varphi(x)\cdot \nabla\Delta^{-1} \tvr_\e^\beta\dd x \dd t = J_1 + J_2 
\end{align*}
for $\beta$ positive, sufficiently small. Since $\tvr_\e = 0$ in $D\setminus D_\e$, the first term is zero. In the second term we take $\beta$ so small that $\|\nabla\Delta^{-1} \tvr_\e^\beta\|_{L^\infty((0,T)\times D)} \lesssim 1$ and then 
\begin{align*}
|J_2| \lesssim \|P_\e\|_{L^{\frac53 - \frac1\gamma}((0,T) \times D)} |D\setminus D_\e|^\frac{2\gamma-3}{5\gamma-3}  \to 0
\end{align*}
for $\e \to 0$. Realizing this fact, we may proceed as in the standard existence proof and send $\e \to 0$  to achieve the effective viscous flux identity
\begin{align*} 
\overline{\rho^{\gamma + \beta}} -\Big(\frac 43 \mu + \eta\Big) \overline{\rho^\beta \div \vu} = \overline{\rho^{\gamma}} \, \overline{\rho^\beta} - \Big(\frac 43 \mu + \eta\Big) \overline{\rho^\beta} \div \vu \qquad \text{ a.e. in } (0,T)\times D.
\end{align*}
This identity, combined with the renormalized form of the continuity equation, yields by standard technique the strong convergence of the density which concludes the proof of Theorem~\ref{thm1}.

\section{Heat conducting fluids}\label{sec:NSF}
In this section, we briefly investigate the homogenization for the case of the full Navier-Stokes-Fourier system, given by
\begin{align}\label{NSF}
\begin{cases}
\del_t \rho_\e + \div(\rho_\e \vu_\e)=0 & \text{in } (0,T)\times D_\e,\\
\del_t(\rho_\e \vu_\e) + \div(\rho_\e \vu_\e \otimes \vu_\e) + \nabla p(\rho_\e, \vartheta_\e) = \div \ms(\vartheta_\e, \nabla \vu_\e) + \rho_\e \vc f & \text{in } (0,T)\times D_\e,\\
\del_t(\rho_\e s(\rho_\e, \vartheta_\e)) + \div(\rho_\e s(\rho_\e, \vartheta_\e) \vu_\e) + \div \frac{\vc q_\e}{\vartheta_\e} = \sigma_\e & \text{in } (0,T)\times D_\e,\\
\vu_\e=0, \ \vc q_\e \cdot \vc n = 0 & \text{on } (0,T)\times \del D_\e,\\
\rho_\e(0,\cdot)=\rho_{\e, 0},\ (\rho_\e\vu_\e)(0,\cdot)=\vc m_{\e, 0},\ \vartheta_\e(0,\cdot) = \vartheta_{\e, 0} & \text{in } D_\e.
\end{cases}
\end{align}
Here, $\rho_\e$, $\vartheta_\e$, and $\vu_\e$ denote the fluid's density, temperature, and velocity, respectively, $p(\rho,\vartheta)=\rho^\gamma + \rho \vartheta + \vartheta^4$ for some $\gamma>\frac32$, $\ms(\vartheta, \nabla \vu)$ is the Newtonian viscous stress tensor of the form
\begin{align}\label{S}
\ms(\vartheta, \nabla\vu)=\mu(\vartheta) \Big(\nabla\vu + \nabla\vu^T-\frac23 \div(\vu)\mathbb{I} \Big)+\eta(\vartheta) \div(\vu)\mathbb{I},
\end{align}
and $\vc f\in L^\infty((0,T)\times D)$ is given. The {\cre specific} entropy corresponding to the given form of the pressure is
\begin{align*}
s(\rho,\vartheta) = \ln\vartheta - \ln\varrho + \frac{4 \vartheta^3}{\varrho}.
\end{align*}
Moreover, the heat flux vector is given by Fourier's law
\begin{align*}
\vc q_\e = - \kappa(\vartheta_\e) \nabla \vartheta_\e,
\end{align*}
and we assume the transport coefficients $\mu, \eta, \kappa$ to be continuously differentiable functions on $[0,\infty)$ with
\begin{align}\label{trkoeff}
\begin{split}
1+\vartheta &\lesssim \mu(\vartheta), \ |\mu'|\lesssim 1,\\
0 &\leq \eta(\vartheta) \lesssim 1+\vartheta,\\
1+\vartheta^m &\lesssim \kappa(\vartheta) \lesssim 1+ \vartheta^m \ \text{for some} \ m>2.
\end{split}
\end{align}

The measure $\sigma_\e$ is called entropy production rate and is assumed to satisfy
\begin{align*}
\sigma_\e \geq \frac{1}{\vartheta_\e} \Big( \ms(\vartheta_\e, \nabla \vu_\e) - \frac{\vc q_\e \cdot \nabla \vartheta_\e}{\vartheta_\e} \Big) \ \text{in the sense of measures}.
\end{align*}

For existence of weak solutions and further details about the homogenization of system \eqref{NSF}, we refer to \cite{PokornySkrisovski2021}. To shorten the computations, we will just focus on how one gets their assumption on the adiabatic exponent from $\gamma>6$ to $\gamma>3$.

Since most of the estimates to do are similar to the ones obtained for the Navier-Stokes system \eqref{NSE}, we will just state the weak form of the momentum equation, which essentially stays as in the case of constant temperature:
\begin{align}\label{wkMomTemp}
\begin{split}
&\int_0^T\int_{D_\e} \rho \vu \cdot \del_t \phi \dd x \dd t + \int_0^T\int_{D_\e} \rho \vu\otimes \vu : \nabla \phi \dd x \dd t + \int_0^T\int_{D_\e} p(\rho,\vartheta) \div \phi \dd x \dd t \\
&- \int_0^T\int_{D_\e} \ms(\vartheta, \nabla \vu):\nabla \phi \dd x \dd t + \int_0^T\int_{D_\e} \rho \vc f \cdot \phi \dd x \dd t = -\int_{D_\e} \vc m_0 \cdot \phi(0,\cdot) \dd x.
\end{split}
\end{align}

To get a similar pressure decomposition as in Section~\ref{sec32}, we define again a pressure extension $P_\e$ by
\begin{align*}
\int_0^T \langle \nabla P_\e, \phi\rangle_{\Test', \Test(D)} \dd t = \int_0^T \langle \nabla p(\rho_\e, \vartheta_\e), R_\e \phi \rangle_{\Test', \Test(D_\e)} \dd t \quad \text{for any } \phi \in \Test((0,T) \times D),
\end{align*}
use equation \eqref{wkMomTemp}, and focus on the only new term involving $\ms(\vartheta, \nabla \vu)$. From \cite[Remark after Proposition~2.4]{PokornySkrisovski2021}, we have the uniform temperature estimate
\begin{align*}
\| \vartheta_\e \|_{L^m(0,T;L^{3m}(D_\e))} \lesssim 1.
\end{align*}
Recalling $R_\e$ as the restriction operator from Theorem~\ref{thm:Restr}, the form of the stress tensor $\ms$ in \eqref{S}, and the assumptions on the transport coefficients \eqref{trkoeff}, we get
\begin{align*}
&\bigg| \int_0^T \int_{D_\e} \ms(\vartheta_\e, \nabla \vu_\e) : \nabla R_\e \phi \dd x \dd t \bigg|\\
&\lesssim \|\vartheta_\e\|_{L^m(0,T;L^{3m}(D_\e))} \|\nabla \vu_\e\|_{L^2(0,T;L^2(D_\e))} \|\nabla R_\e\phi\|_{L^\frac{2m}{m-2}(0,T;L^\frac{6m}{3m-2}(D_\e))}\\
&\lesssim \Big(1+ \e^{(3-\frac{6m}{3m-2})\alpha - 3} \Big)^\frac{3m-2}{6m} \|\nabla \phi\|_{L^\frac{2m}{m-2}(0,T;L^\frac{6m}{3m-2}(D))},
\end{align*}
which is uniform as long as
\begin{align*}
m>2 \ \text{and} \ \alpha>\frac{3}{3-\frac{6m}{3m-2}} = \frac{3m-2}{m-2}.
\end{align*}
Note that this is the same restriction as in \cite{LuPokorny2021, PokornySkrisovski2021}. The remaining estimates are the same as in Section~\ref{sec32}, leading to a pressure decomposition of the form
\begin{equation*}
P_\e \ \text{uniformly bounded in} \ L^1(0,T;L^{\min\{2, \frac{3\gamma}{\gamma+3}\}}(D)) + \widehat{W}^{-1,2}(0,T;L^2(D)) + L^\frac{2m}{m+2}(0,T;L^\frac{6m}{3m+2}(D)).
\end{equation*}
Note especially that $m>2$ leads to $\frac{6m}{3m+2}>\frac32$, in accordance with Remark~\ref{rem1}.

Repeating now the steps done in \cite[Sections~2 and 3]{PokornySkrisovski2021}, we have shown:
\begin{theorem}\label{thm2}
Let $D\subset\R^3$ be a bounded domain with smooth boundary, $K_\e \subset D$ comply with \eqref{defKeps}, and $D_\e$ be defined as in \eqref{defDeps}. Let $(\rho_\e, \vartheta_\e, \vu_\e)$ be a sequence of renormalized finite energy weak solutions to system \eqref{NSF} emanating from the initial data $(\rho_{\e, 0}, \vartheta_{\e, 0}, \vc m_{\e, 0})$ satisfying the assumptions given in \cite[Theorem~1.2]{PokornySkrisovski2021}. Then, there exists a subsequence (not relabelled) such that
\begin{align*}
\tilde{\rho}_\e &\weak^* \rho \text{ weakly-$\ast$ in } L^\infty(0,T;L^\gamma(D)) \ \text{and weakly in} \ L^{\frac53 \gamma - 1}((0,T) \times D),\\
\tilde{\vu}_\e &\weak \vu \text{ weakly in } L^2(0,T;W_0^{1,2}(D)),\\
\tilde{\vartheta}_\e &\weak \vartheta \text{ weakly in } L^m(0,T;L^{3m}(D))\cap L^2(0,T;W^{1,2}(D)) \ \text{and weakly-$*$ in} \ L^\infty(0,T;L^4(D)),
\end{align*}
where $(\rho,\vartheta,\vu)$ is a renormalized weak solution to system \eqref{NSF} in the domain $(0,T)\times D$ with initial conditions $\rho(0,\cdot)=\rho_0$, $\vartheta(0,\cdot)=\vartheta_0$, and $(\rho\vu)(0,\cdot)=\vc m_0$, provided
\begin{align*}
\gamma>3,\ m>2, \ \text{and} \ \alpha>\max \Big\{\frac{2\gamma-3}{\gamma-3}, \frac{3m-2}{m-2} \Big\}. 
\end{align*}
\end{theorem}

\begin{rem}
Often, due to physics, the value of $m$ is taken to be equal to $3$ which is connected with the term $\vartheta^4$ in the pressure (the power is connected with the so-called Stefan--Boltzmann law of radiation, see \cite[Section 2.2.3]{FeireislNovotny2009singlim}; the growth in the molecular part of the heat conductivity is usually assumed to be slower). Then, the restriction on $\alpha$ reduces to 
\begin{align*}
\alpha>\max \Big\{ \frac{2\gamma-3}{\gamma-3}, 7 \Big\},
\end{align*}
i.e., for $\gamma > \frac {18}5$ the bound on $\alpha$ is independent of $\gamma$.
\end{rem}

\section{A different approach via special cut-off functions}\label{sec:Bravin}
In this section, we want to give another proof of Theorem~\ref{thm1}. To this end, we will not take advantage of the pressure decomposition in Section~\ref{sec32} but rather use special cut-off functions that differ from those in Lemma~\ref{lem:Cutoff}. To define an appropriate cut-off function for multiple holes in the whole of $D$, we follow an idea of Bravin in \cite{Bravin2022} for a single hole in $\R^2$, a generalization of which was already used in \cite{NecasovaOschmann2023}. {\cre For $0<\eta < R$, we set}
\begin{align}\label{zeta3}
\zeta_{\eta,R}(r)=\begin{cases}
1 & \text{if } 0\leq r<\eta,\\
\frac{1/R-1/r}{1/R-1/\eta} & \text{if } \eta\leq r<R,\\
0 & \text{else},
\end{cases}
\end{align}
and define
\begin{align*}
\cre \mathfrak{y}_\e^0(r)=\zeta_{\e^\alpha,\e^{1+\delta}}(r)
\end{align*}
for some $\delta>0$ such that $2\e^\alpha<\e^{1+\delta}$. With {\cre$r=|z|$}, an easy calculation leads to
\begin{align}\label{est0}
\begin{split}
\|\mathfrak{y}_\e^0\|_{L^\infty(\R^3)}+\|z_i \nabla\mathfrak{y}_\e^0\|_{L^\infty(\R^3)} &\lesssim 1,\\
\|\nabla \mathfrak{y}_\e^0\|_{L^q(\R^3)}^q + \|z_i \nabla^2\mathfrak{y}_\e^0\|_{L^q(\R^3)}^q &\lesssim
\e^{(3-q)\alpha} \begin{cases}
|\e^{(3-2q)(1+\delta-\alpha)} - 1| & \text{if } q\neq \frac32,\\
|\log(\e^{1+\delta-\alpha})| & \text{if } q=\frac32,
\end{cases}
\end{split}
\end{align}
{\cre where $z_i$ denotes the $i$-th component of a coordinate $z \in \R^3$.} We set $\mathfrak{y}_\e^i(x)=\mathfrak{y}_\e^0(|x-x_i(\e)|)$, and define the matrix-valued cut-off function
\begin{align*}
\Phi_\e = \mathbb{I} - \sum_{i=1}^{N(\e)} \Big(\mathfrak{y}_\e^i\mathbb{I} + \nabla \mathfrak{y}_\e^i \wedge \mathbf{T} [\, \cdot - x_i(\e)]\Big).
\end{align*}
Here,
\begin{align*}
\mathbf{T} [x]=\frac12 \begin{pmatrix}
0 & x_3 & -x_2\\
-x_3 & 0 & x_1\\
x_2 & -x_1 & 0
\end{pmatrix},
\end{align*}
and the cross-product has to be understood in the sense of tensors, that is,
\begin{align*}
(a\wedge \mathbf{T})\cdot b = a\wedge (\mathbf{T} \cdot b) \quad \forall a,b\in\R^3
\end{align*}
with
\begin{align*}
a \wedge b = (a_2b_3-a_3b_2, a_3 b_1-a_1b_3, a_1b_2-a_2b_3)
\end{align*}
for $a,b \in \R^3$, and $(\mathbf{T} \cdot b)_i = \sum_{j=1}^3 \mathbf{T}_{ij}b_j$. 
We remark that the matrix $\mathbf{T}$ is chosen such that $\nabla \wedge \mathbf{T} = \mathbb{I}$.
Note especially that
\begin{align*}
\div \Phi_\e = - \sum_{i=1}^{N(\e)} \div \big(\nabla\wedge \big[\mathfrak{y}_\e^i \mathbf{T} [x-x_i(\e)] \big]\big)=0.
\end{align*}
We summarize the properties of $\Phi_\e$ in the following

\begin{lemma}\label{lemPhi}
The function $\Phi_\e$ fulfils
\begin{align*}
\Phi_\e &\in W^{1,q}(D)\cap L^\infty(D) \text{ for any } q\geq 1,\\
\Phi_\e &=0 \text{ on } D\setminus D_\e,\\
\Phi_\e &=\mathbb{I} \text{ on } D\setminus \bigcup_{i=1}^{N(\e)} B_{\e^{1+\delta}}(x_i(\e)).
\end{align*}
Moreover, $\|\Phi_\e\|_{L^\infty(D)}\lesssim 1$, and for any $1\leq q< \infty$,
\begin{align*}
\|\Phi_\e-\mathbb{I}\|_{L^q(D)}^q &\lesssim \e^{3\delta}, &&\|\nabla\Phi_\e\|_{L^q(D)}^q \lesssim \e^{(3-q)\alpha - 3} \begin{cases}
|\e^{(3-2q)(1+\delta-\alpha)} - 1| & \text{if } q\neq \frac32,\\
|\log(\e^{1+\delta-\alpha})| & \text{if } q=\frac32.
\end{cases}
\end{align*}
In turn, for any $\psi\in C_c^\infty(D;\R^3)$ and any $q>2$,
\begin{align*}
\|\nabla (\Phi_\e \psi)-\Phi_\e \nabla\psi\|_{L^q(D)}^q &\lesssim \e^{(3-q)\alpha - 3} \|\psi\|_{L^\infty(D)}^q.
\end{align*}
\end{lemma}
\begin{proof}
Once noticing that the holes are disjoint and their number in $D$ grows like $\e^{-3}$, we immediately get the desired estimates from \eqref{est0}. The estimate on $\nabla(\Phi_\e \psi)-\Phi_\e \nabla\psi$ is a direct consequence of H\"older's inequality
\begin{align*}
\|\nabla (\Phi_\e \psi)-\Phi_\e \nabla\psi\|_{L^q(D)}^q &= \|(\nabla (\Phi_\e \vc e_1) \psi , \nabla (\Phi_\e \vc e_2) \psi, \nabla (\Phi_\e \vc e_3) \psi)\|_{L^q(D)}^q \\
&\leq \|\nabla\Phi_\e\|_{L^q(D)}^q \|\psi\|_{L^\infty(D)}^q
\lesssim \e^{(3-q)\alpha - 3} \|\psi\|_{L^\infty(D)}^q.
\end{align*}
\end{proof}

{\cre
\begin{rem}
As one can see, the function $\zeta_{\eta,R}$ defined in \eqref{zeta3} is not differentiable, especially not smooth, and so isn't the matrix $\Phi_\e$. This issue can be resolved in introducing another cut-off function that has a smoothing effect; since this procedure is merely technical, we will not go into details and refer to \cite[Appendix C]{Bravin2022}.
\end{rem}
}

\subsection{Uniform bounds}
The uniform bounds are exactly the same as in Section~\ref{sec:Bds} since only the Bogovski\u{\i} operator is involved, no cut-off.

\subsection{Convergence proof}
As mentioned in Section~\ref{sec:Conv}, we have that
the extended functions $\tilde{\rho}_\e$ and $\tilde{\vc u}_\e$ fulfil
\begin{align*}
\del_t \tilde{\rho}_\e + \div(\tilde{\rho}_\e \tilde{\vc u}_\e)=0 \text{ in } \mathcal{D}'((0,T)\times \R^3),
\end{align*}
as well as, after the limit passage,
\begin{align*}
\del_t \rho + \div(\rho \vc u)=0 \text{ in } \mathcal{D}'((0,T)\times \R^3),
\end{align*}
together with the corresponding renormalized forms. 
Moreover, by the arguments from Section~\ref{sec:Conv} and the strong convergence of $\Phi_\e$ from Lemma~\ref{lemPhi},
\begin{align}\label{eq:CV}
\begin{split}
\Phi_\e^T \tilde{\rho}_\e \tilde{\vc u}_\e &\to \rho \vc u \text{ in } C(0,T;L_{\rm weak}^r(D)), \quad r< \frac{2\gamma}{\gamma+1},\\
\Phi_\e^T \tilde{\rho}_\e \tilde{\vc u}_\e \otimes \tilde{\vc u}_\e &\to \rho \vc u \otimes \vc u \text{ in } \mathcal{D}'((0,T)\times D).
\end{split}
\end{align}

For $\phi \in C_c^\infty([0,T)\times D)$, we use $\Phi_\e \phi \in C_c^\infty([0,T) \times D_\e)$ as a proper test function. Since $\Phi_\e = 0$ on the holes, we can prolong all functions by zero to the whole of $D$ and obtain
\begin{align*}
0&=\int_D \tilde{\vc m}_{\e0}\cdot  \Phi_\e \phi(0,\cdot) \dd x + \int_0^T\int_D \tilde{\rho}_\e\tilde{\vc u}_\e \cdot \Phi_\e \del_t\phi \dd x \dd t + \int_0^T\int_D \tilde{\rho}_\e \tilde{\vc u}_\e \otimes \tilde{\vc u}_\e : \nabla (\Phi_\e \phi) \dd x \dd t\\
&\quad + \int_0^T\int_D \tilde{\rho}_\e^\gamma \div(\Phi_\e \phi) \dd x \dd t - \int_0^T \int_D \mathbb{S}(\nabla\tilde{\vc u}_\e):\nabla(\Phi_\e \phi) \dd x \dd t + \int_0^T\int_D \tilde{\rho}_\e \vc f\cdot\Phi_\e \phi \dd x \dd t\\
&= \sum_{j=1}^6 I_j.
\end{align*}

By the strong convergence of $\Phi_\e \to \mathbb{I}$ in any $L^q(D)$ and the convergences obtained in \eqref{eq:CV}, we can easily pass to the limit in $I_1, I_2,$ and $I_6$. Furthermore, $\div \Phi_\e = 0$, so the pressure integral reads
\begin{align*}
\int_0^T \int_D \tilde{\rho}_\e^\gamma \div (\Phi_\e \phi) \dd x \dd t = \int_0^T \int_D \tilde{\rho}_\e^\gamma \Phi_\e:\nabla \phi \dd x \dd t,
\end{align*}
and by the same argument, this converges to $\int_0^T \int_D \overline{\rho^\gamma}\div \phi \dd x \dd t$.\\

For the diffusive part $I_5$, we have
\begin{align*}
\int_0^T\int_D \mathbb{S}(\nabla\tilde{\vc u}_\e):\nabla(\Phi_\e\phi) \dd x \dd t &= \int_0^T\int_D \mathbb{S}(\nabla\tilde{\vc u}_\e):(\Phi_\e \nabla \phi) \dd x \dd t \\
&\quad + \int_0^T\int_D \mathbb{S}(\nabla\tilde{\vc u}_\e):(\nabla(\Phi_\e \phi)-\Phi_\e\nabla\phi) \dd x \dd t.
\end{align*}
The latter term converges to zero due to
\begin{align*}
\bigg|\int_0^T\int_D \mathbb{S}(\nabla \tilde{\vc u}_\e) : (\nabla(\Phi_\e \phi)-\Phi_\e\nabla\phi) \dd x \dd t \bigg| &\lesssim \|\nabla \tilde{\vc u}_\e\|_{L^2(0,T;L^2(D))} \|\nabla(\Phi_\e \phi)-\Phi_\e\nabla\phi\|_{L^\infty(0,T;L^2(D))}\\
& \lesssim \|\nabla \Phi_\e\|_{L^2(D)} \|\phi\|_{L^\infty((0,T)\times D)} \lesssim \e^\frac{\alpha-3}{2}\|\phi\|_{L^\infty((0,T)\times D)}.
\end{align*}
Together with the strong convergence of $\Phi_\e \to \mathbb{I}$ in $L^2(D)$ and the weak convergence of $\nabla\tilde{\vc u}_\e \weak \nabla\vc u$ in $L^2((0,T)\times D)$, we deduce
\begin{align*}
\int_0^T\int_D \mathbb{S}(\nabla\tilde{\vc u}_\e):\nabla(\Phi_\e\phi) \to \int_0^T\int_D \mathbb{S}(\nabla\vc u):\nabla\phi.
\end{align*}

The remaining convective part $I_3$ is handled similarly as
\begin{align*}
\int_0^T\int_D \tilde{\rho}_\e \tilde{\vc u}_\e \otimes \tilde{\vc u}_\e : \nabla (\Phi_\e \phi) \dd x \dd t &= \int_0^T\int_D \tilde{\rho}_\e \tilde{\vc u}_\e \otimes \tilde{\vc u}_\e : (\Phi_\e \nabla \phi) \dd x \dd t\\
&\quad  + \int_0^T\int_D \tilde{\rho}_\e \tilde{\vc u}_\e \otimes \tilde{\vc u}_\e : (\nabla(\Phi_\e\phi)-\Phi_\e\nabla\phi) \dd x \dd t\\
&= \int_0^T\int_D \Phi_\e^T \tilde{\rho}_\e \tilde{\vc u}_\e \otimes \tilde{\vc u}_\e : \nabla \phi \dd x \dd t\\
&\quad  + \int_0^T\int_D \tilde{\rho}_\e \tilde{\vc u}_\e \otimes \tilde{\vc u}_\e : (\nabla(\Phi_\e\phi)-\Phi_\e\nabla\phi) \dd x \dd t.
\end{align*}
The latter term vanishes due to the embedding $W_0^{1,2}(D)\subset L^6(D)$. Indeed, we get with $\gamma>3$ and the uniform bounds on $\rho_\e$ and $\vu_\e$
\begin{align*}
&\bigg| \int_0^T\int_D \tilde{\rho}_\e\tilde{\vc u}_\e\otimes \tilde{\vc u}_\e : (\nabla(\Phi_\e\phi)-\Phi_\e\nabla\phi) \bigg|\\
&\leq \|\tilde{\rho}_\e\|_{L^\infty(0,T;L^\gamma(D))} \|\tilde{\vc u}_\e\|_{L^2(0,T;L^6(D))}^2 \|\nabla(\Phi_\e\phi)-\Phi_\e\nabla\phi\|_{L^\infty(0,T;L^\frac{3\gamma}{2\gamma-3}(D))}\\
&\lesssim \e^{\big(3-\frac{3\gamma}{2\gamma-3}\big)\alpha-3} \|\phi\|_{L^\infty((0,T)\times D)},
\end{align*}
where as before $(3-\frac{3\gamma}{2\gamma-3})\alpha-3 = \alpha\frac{3(\gamma-3)}{2\gamma-3} - 3 > 0$ by assumption \eqref{cond1}. Hence, by $\Phi_\e^T \tilde{\rho}_\e \tilde{\vc u}_\e \otimes \tilde{\vc u}_\e \to \rho \vc u \otimes \vc u$ in $\mathcal{D}'((0,T)\times D)$, we obtain
\begin{align*}
\int_0^T\int_D \tilde{\rho}_\e \tilde{\vc u}_\e \otimes \tilde{\vc u}_\e : \nabla (\Phi_\e \phi) \dd x \dd t \to \int_0^T\int_D \rho\vc u\otimes\vc u : \nabla \phi \dd x \dd t.
\end{align*}
This finishes the alternative proof of Theorem~\ref{thm1}.

\section{Concluding remarks: dimensional optimality}\label{sec:Concl}
In this paper, we have shown that solutions $(\rho_\e, \vu_\e)$ of system \eqref{NSE} converge to solutions $(\rho, \vu)$ of the same system in $(0,T) \times D \subset (0,T) \times \R^d$, $d=3$, as long as the holes are in a certain sense small, and the adiabatic exponent $\gamma>3=d$. In \cite{NecasovaOschmann2023}, a similar result was shown for the two-dimensional case as long as $\gamma>2=d$. A natural question to ask now is whether such coincidences between the reachable exponent $\gamma$ and the dimension $d$ hold in general. Indeed, following the proofs of Lemma~\ref{lem:RegDens}, Theorem~\ref{thm:Bog}, Theorem~\ref{thm:Restr}, and Lemma~\ref{lemPhi}, one easily finds that all the assertions stated there are still valid for any $d\geq 3$, provided the crucial exponent $\e^{(3-q)\alpha-3}$ is changed to $\e^{(d-q)\alpha-d}$. In terms of capacity, this scaling is optimal since any hole would contribute a capacity of $\e^{(d-q)\alpha}$ to the system, and there are $\e^{-d}$ of them (see also \cite[Remark~2.4]{Lu2021}). Applying the same pressure decomposition as in Section~\ref{sec32} and setting $2^\ast = 2d/(d-2)$, we find for the (most restrictive) convective term
\begin{align*}
&\int_0^T \int_{D_\e} \rho_\e \vu_\e \otimes \vu_\eps : \nabla R_\e \phi \dd x \dd t\\
&\lesssim \|\rho_\e\|_{L^\infty(0,T;L^\gamma(D_\e))} \|\vu_\e\|_{L^2(0,T;L^{2^\ast}(D_\e))}^2 \|\nabla R_\e \phi\|_{L^\infty(0,T;L^\frac{2^\ast \gamma}{2^\ast(\gamma-1)-2\gamma}(D_\e))}\\
&\lesssim \bigg( 1 + \e^{(d-\frac{2^\ast \gamma}{2^\ast(\gamma-1)-2\gamma})\alpha - d} \bigg)^\frac{2^\ast(\gamma-1)-2\gamma}{2^\ast \gamma},
\end{align*}
which is uniformly bounded as long as
\begin{align*}
\alpha > \frac{d}{d-\frac{2^\ast \gamma}{2^\ast(\gamma-1)-2\gamma}} = \frac{2\gamma - d}{\gamma-d}, \quad \gamma > d.
\end{align*}
Similarly, the diffusive part leads to
\begin{align*}
\int_0^T \int_{D_\e} \ms(\nabla \vu_\e) : \nabla R_\e \phi \dd x \dd t \lesssim \|\nabla R_\e \phi\|_{L^2(0,T;L^2(D_\e))} \lesssim \big( 1+ \e^{(d-2)\alpha - d} \big)^\frac12,
\end{align*}
which is uniform as long as $\alpha>\frac{d}{d-2}$. Finally, we arrive at the conditions
\begin{align*}
\alpha>\max \left\{ \frac{d}{d-2}, \frac{2\gamma - d}{\gamma-d} \right\},\quad \gamma>d \geq 3.
\end{align*}
Note that the last term in the maximum wins precisely if $d=3$ and $\gamma \leq 6$, or $d\geq 4$. Further, it is remarkable that the lower bound for $\gamma$ connects to the dimension of the underlying space as {\cre $\gamma>\mathfrak{D}-1$, where $\mathfrak{D}=d+1$} corresponds to the space-time dimension of $(0,T)\times \R^d \sim \R^{d+1}$. In view of this, it seems that one cannot go below this bound. Moreover, similar considerations for the stationary case, where the known values are $\gamma>1$ if $d=2$ (see \cite{NecasovaPan2022}), and $\gamma>2$ if $d=3$ (see \cite{DieningFeireislLu2017}), lead to
\begin{align*}
\alpha > \max \left\{\frac{d}{d-2}, \frac{2\gamma-d}{\gamma-(d-1)} \right\}, \quad \gamma>d-1 \geq 2.
\end{align*}
{\cre In both cases, the optimal lower bound for $\gamma$ is then indeed $\gamma>\mathfrak{D}-1$, where $\mathfrak{D}$ is the space-time dimension given by $\mathfrak{D}=d$ in the steady case, and $\mathfrak{D}=d+1$ in the unsteady case.}

\section*{Acknowledgement}
{\it F. O. and M. P. have been supported by the Czech Science Foundation (GA\v CR) project 22-01591S. The Institute of Mathematics, CAS is supported by RVO:67985840.}


\bibliographystyle{amsalpha}

\end{document}